\documentclass[11pt,reqno,a4paper]{amsart}
\usepackage{a4}
\usepackage{mathrsfs}
\usepackage{dsfont}
\usepackage{hyperref}
\usepackage{url}            
\usepackage{doi}

\usepackage{pgfplots}
\pgfplotsset{compat=1.18} 
\usepackage{graphicx} 
\usepackage{float}
\usepackage{multirow}%
\usepackage{amsmath,amssymb,amsfonts}%
\usepackage{amsthm}%
\usepackage[title]{appendix}%
\usepackage{xcolor}%
\usepackage{textcomp}%
\usepackage{manyfoot}%
\usepackage{booktabs}%
\usepackage{algorithm}%
\usepackage{algorithmicx}%
\usepackage{algpseudocode}%
\usepackage{listings}%
\usepackage{circuitikz}
\usepackage{tikz}
\usepackage{bm}
\usepackage{xspace}
\usepackage{subcaption}
\usepackage{mathtools}
\usepackage{fix-cm}
\usepackage{lmodern}
\usepackage{hyperref} 

\newtheorem{theorem}{Theorem}
\newtheorem{proposition}[theorem]{Proposition}%
\newtheorem{lemma}{Lemma}%
\newtheorem{definition}{Definition}%
\newcommand{\Syssymb}{\Sigma}
\newcommand{\V}{\mathbb{V}}
\newcommand{\ROMnota}[1]{#1_{red}}
\newcommand{\ROMnotaZero}[1]{#1_{red,0}}
\newcommand{\SoBmat}{\Gamma}

\newcommand{\transpose}{\mathsf{T}}
\newcommand{\ddt}{\frac{d}{dt}}

\newcommand{\Linv}{-\mathsf{L}}

\newcommand{\portsize}{p}
\newcommand{\FOMsize}{n}
\newcommand{\ROMsize}{k}
\newcommand{\fresize}{M}

\newcommand{\pH}{\textsf{pH}\xspace}

\newcommand{\FOM}{\textsf{FOM}\xspace}
\newcommand{\ROM}{\textsf{ROM}\xspace}

\newcommand{\ROMs}{\textsf{ROM}s\xspace}
\newcommand{\MOR}{\textsf{MOR}\xspace}

\newcommand{\LTI}{\textsf{LTI}\xspace}
\newcommand{\MSD}{\textsf{MSD}\xspace}
\newcommand{\LRCR}{\textsf{LRCR}\xspace}
\newcommand{\SO}{\textsf{SO}\xspace}
\newcommand{\MarkSize}{3pt}
\newcommand{\PlotLineWidth}{1pt}
\newcommand{\cootrans}[1]{\hat{#1}}

\usepackage{amsaddr}
\makeatletter
\renewcommand{\email}[2][]{%
	\ifx\emails\@empty\relax\else{\g@addto@macro\emails{,\space}}\fi%
	\@ifnotempty{#1}{\g@addto@macro\emails{\textrm{(#1)}\space}}%
	\g@addto@macro\emails{#2}%
}
\makeatother

\title[Symplectic Model Order Reduction of port-Hamiltonian Systems]{Symplectic Model Order Reduction of port-Hamiltonian Systems}

\author[S.~Glas \and M.~Mamunuzzaman \and H.Mu \and H.~Zwart]{Silke Glas $^{\dagger}$
	\and Mir Mamunuzzaman$^{\dagger}$
	\and Hongliang Mu $^{\dagger}$
	\and Hans Zwart$^{\dagger,\ddagger}$}

\address{${}^{\dagger}$ Department of Applied Mathematics, University of Twente, P.O. Box 217, 7500 AE Enschede, The Netherlands\\
(e-mail: {s.m.glas@utwente.nl,
	mirmamunuzzaman@gmail.com,
	h.l.mu@utwente.nl,
	h.j.zwart@utwente.nl
	})
}
\address{${}^{\ddagger}$ Department of Mechanical Engineering, Eindhoven University of Technology, P.O. Box 513, 5600 MB Eindhoven, The Netherlands\\
(e-mail: {h.j.zwart@tue.nl
})
}

\thanks{HZ and MM have received funding from the European Union's Horizon 2020 research and innovation program under the Marie Sklodowska-Curie grant agreement No.\ 765579.}

\begin{document}
	\maketitle
	
	\begin{abstract}
		This work proposes a novel structure-preserving model order reduction (\MOR) method for linear, time-invariant port-Hamiltonian (\pH) systems. Our goal is to construct a reduced order \pH system, which can still be interpreted in the physical domain of the full order model. By this we mean, that if an electrical circuit is the initial high-dimensional \pH system, we want the reduced order model to be still interpretable as an electronic circuit. In the case of the well-known mass spring damper (\MSD) system, there are \MOR methods available, which already guarantee the preservation of this particular structure. Moreover, we show that our new structure-preserving \MOR method, which is based on symplectic \MOR methods, will recover the known second-order Arnoldi method in the case of  \MSD systems. However, for the example of an electrical circuit \pH model (and more models of similar block structure), our method yields a novel model reduction method. We present numerical results on the aforementioned electronic circuit model, highlighting the advantages of the proposed method.
	\end{abstract}
	
	{\small 
		
		{\bf AMS-classification:} 34C20, 65P10, 70H33, 93B10, 93B1, 93C05.
	}

	\section{Introduction}
	\label{sec:intro} 
	
	Port-Hamiltonian (\pH) modelling is a generalisation of the classical Hamiltonian framework. Among many other reasons, the \pH framework is very useful in engineering problems since it combines the classical Hamiltonian approach with a powerful framework for modelling and simulation of many classes of open physical systems, e.g., network modelling or electronic circuits.
	
	In this paper, we consider structure-preserving model order reduction (\MOR) for such linear time-invariant \pH systems. These are systems of the following form\footnote{We formulate our systems in the complex domain, since even though many systems have real parameters, we interpolate at complex values.} \footnote{Since port-Hamiltonian systems often are of even state dimension, we assume this from the start.}
	\begin{equation}
		\label{equ:pH_sys}
		\dot{x}(t) = (J - R){\mathcal H} x(t) + B u(t),\quad y(t) = B^{*} {\mathcal H} x(t),
	\end{equation}
	where $x(t) \in {\mathbb C}^{2 \FOMsize},u(t) \in {\mathbb C}^{\portsize}$, $y(t) \in {\mathbb C}^{\portsize}$ denote the state, input, and output respectively, and the matrices satisfy $-J=J^{*} \in {\mathbb C}^{2\FOMsize \times 2 \FOMsize}$, 
	$0\leq R=R^{*} \in {\mathbb C}^{2\FOMsize \times 2\FOMsize}$, $B \in {\mathbb C}^{2\FOMsize \times \portsize}$, and $0 < {\mathcal H}={\mathcal H}^{*} \in {\mathbb C}^{2 \FOMsize \times 2\FOMsize}$. 
	The aim is to find a system of the same form, but with a smaller state space dimension, such that the two transfer functions from the full-order and reduced-order systems are equal at a prescribed set of frequency points.
	
	In many control design methods, like LQ and $H^{\infty}$, the controller will be of the same order as the original system, which motivates the need for a reduced order model (\ROM) capturing the essential behaviour of the original system.
	Model order reduction is a well-established field with a long history, \cite{OptimalHankelNorm1984}, \cite{MoorePCA}, but \MOR methods for \pH-systems are more recent, see e.g.\ the overview paper \cite{Anto05}. Although the term \pH-system is not explicitly used in the latter overview paper, since many \pH-systems have a positive real transfer function, the results presented there can be used for \pH-systems. 
	Although the first results on \MOR of \pH-systems/passive systems goes back almost two decades, recently there has been many new approaches to these questions. For example, moment matching \MOR of \pH-systems is done in \cite{Polyuga_MOR_PHS} and \cite{StructurePreserving_PHS}. Approximation of the underlying Dirac structure is the approach taken in  \cite{EffortFlowConstraint_PHS}.
	In \cite{breiten2021passivity}, the authors use the spectral factorisation of the Popov function, i.e., $G(s)+G(-s)^*$ with $G(s)$ being the transfer function of (\ref{equ:pH_sys}), and the singular values of the solutions to the Lur'e equations to construct a reduced \pH model.  In \cite{BrSc25} the authors design a reduced model based on LQG balancing. Further, in \cite{SOBMOR}, the \MOR methods which is based on parameter optimization to retain a reduced \pH system. The work \cite{sato2018riemannian} is based on Riemannian optimization and \cite{ionescu2013moment} is a moment matching method for nonlinear \pH systems. 
	
	As a special instance of \pH systems, there are Hamiltonian systems with no dissipation and no input/output term. For this special class of systems, it is important to preserve the underlying symplectic structure in the reduced order model. Known/recent works for this are by \cite{Symplectic, AfkH17} for linear symplectic reduced order models, \cite{BGH_symplectic, sharma2023symplectic} for reduced order models with nonlinear approximation methods and operator inference/system identification in \cite{ShaWK22}.	
	In this article in particular, we combine the techniques of \cite{Symplectic} with the transfer function interpolation of \cite{BookMOR} to derive a novel \MOR method, extending the class of \pH systems it can be applied to. 
	
	Although the transfer function of our reduced model equals the transfer function of the original system in a prescribed set of frequency points, we will not use the Loewner basis,  \cite{NewPassivityMOR}, but use an approach much closer to the Krylov-based \MOR, \cite{IRKA2005}, \cite{RKIA2006}, \cite{Gugercin2012IRKA}, \cite{PassivityPreserving_LPHS}. The $H^2$-optimality conditions for this approach are given in \cite{H2MOR_LargeSystem}.
	
	With this novel \MOR method based on symplectic \MOR, we can address structure-preservation for a special class of \pH systems: When reformulating a physical model into \pH form \eqref{equ:pH_sys} (if it exists), we often observe that the matrices $J, R$ and ${\mathcal H}$ naturally obtain a special block structure, e.g., 
	\begin{equation}
		\label{eq:pH-block_FOM}
		\begin{array}{rcl}
			\frac{\mathsf{d}}{\mathsf{d} t} 
			\begin{bmatrix} x_1(t) \\ x_2(t) \end{bmatrix} 
			&=&\ 
			\left(\begin{bmatrix} 0_{\FOMsize} & \tilde{J} \\ -\tilde{J}^* & 0_{\FOMsize}  \end{bmatrix}
			- \begin{bmatrix} R_1 & 0_{\FOMsize}\\ 0_{\FOMsize} &R_2 \end{bmatrix}\right)
			\begin{bmatrix} {\mathcal H}_1 & 0_{\FOMsize}\\ 0_{\FOMsize} &{\mathcal H}_2 \end{bmatrix}
			\begin{bmatrix} x_1(t) \\ x_2(t) \end{bmatrix} 
			+ \begin{bmatrix} B_1 \\ B_2 \end{bmatrix}u(t) ,\\
			y(t) &=&\ \begin{bmatrix} B_1^* & B_2^* \end{bmatrix}
			\begin{bmatrix} {\mathcal H}_1 & 0_{\FOMsize}\\ 0_{\FOMsize} &{\mathcal H}_2 \end{bmatrix}
			\begin{bmatrix} x_1(t) \\ x_2(t) \end{bmatrix} ,
		\end{array}
	\end{equation}
	where $\tilde{J}\in{\mathbb C}^{\FOMsize\times\FOMsize}$, $0\leq R_1=R_1^* \in {\mathbb C}^{\FOMsize\times \FOMsize}$, $0\leq R_2=R_2^* \in {\mathbb C}^{\FOMsize\times \FOMsize}$, $0< H_1=H_1^* \in {\mathbb C}^{\FOMsize\times \FOMsize}$, $0< H_2=H_2^* \in {\mathbb C}^{\FOMsize\times \FOMsize}$, $B_1, B_2 \in {\mathbb C}^{\FOMsize\times \ROMsize}$ are inherited from \eqref{equ:pH_sys}, and $0_{\FOMsize} \in {\mathbb C}^{\FOMsize \times \FOMsize}$ denotes the zero matrix.   
	In our reduced-order model (\ROM), we want to keep this block structure as well such that we can guarantee, that the \ROM can still be interpreted as a physical system of the same type as the full order model (\FOM). This results in a novel \MOR approach, which we call \emph{symplectic-\pH}. To the best of our knowledge, such an approach for the class of systems \eqref{eq:pH-block_FOM} has not been introduced, unless special cases as the e.g., mass spring damper system, with $R_1=0_{\FOMsize}$, $B_1=0_{\FOMsize\times\portsize}$ and $\tilde{J}$ being the identity matrix, are considered. 
	
	In particular, this article provides the following contributions: 
	\begin{enumerate}
		\item We introduce symplectic-\pH, a novel \MOR method of \pH systems, based on symplectic \MOR, of block structured form \eqref{eq:pH-block_FOM}, such that the \ROM results in a \pH system of the same form. Moreover, we show that the transfer function of the \FOM and the \ROM are equal at a prescribed set of frequencies (Section \ref{sec:symp_MOR_pH}). 
		\item In the special case of \eqref{eq:pH-block_FOM} being in the form of a \MSD system, we prove that the introduced method recovers the second-order Arnoldi method (\SO-Arnoldi) published in \cite{Bai2005} (Section \ref{sec:MSD}). This reveals, that the \SO-Arnoldi method preserves more structure than previously known.     
		\item Numerical experiments: We present numerical results for an electronic circuit, more specifically an \LRCR circuit. We compare our method to other known methods in the literature and further observe, that the \SO-Arnoldi method yields unstable results in the \ROM (Section \ref{sec:RLC}). 
	\end{enumerate}
	
	The outline of this article is as follows. In Section \ref{Sec:prelim}, we provide some background and results on \MOR of general linear time-invariant systems and \pH systems. Section \ref{sec:symplecticMOR} shows how the underlying symplectic structure can be conserved when performing model reduction. Furthermore, there we introduce and prove that the newly introduced \MOR method keeps the additional block structure in the system matrices. In Section \ref{sec:MSD}, we consider the special case of \MSD systems and prove that our method gives the same reduced system as the \SO-Arnoldi method. For the numerical results in Section \ref{sec:RLC}, we consider an electrical circuit model and show that our method preserves the structure, while the \SO-Arnoldi method may yield  unstable results. 
	
	\section{Preliminaries and problem setting} \label{Sec:prelim}
	
	In this section we summarize the fundamentals of \MOR of finite-dimen\-sional linear time-invariant (\LTI) systems and review the conventional approach to \MOR for this system class in Section \ref{sec:MOR_projection}. Modelling of \pH systems is stated in Section \ref{revPH} and \MOR of \pH systems is recalled in Section \ref{sec:MOR_pH_sys}.
	
	\subsection{Model order reduction by projection}\label{sec:MOR_projection}
	Consider the \LTI system described by the standard state-space model
	\begin{equation}
		\label{eq:FOM}
		{\Syssymb}
		\ :\ \ \ \left\{
		\begin{array}{ll}
			\dot{x}(t) &= A x(t) + B u(t);\ x(0) = x_0;\\
			y(t) &= C x(t) + D u(t),
		\end{array}
		\right. 
	\end{equation}
	where $x(t) \in  {\mathbb C}^{2\FOMsize}$ is the state variable, $ x_0 \in {\mathbb C}^{2\FOMsize}$ is the initial condition and $u(t) \in {\mathbb C}^{\portsize}$ and $y(t) \in {\mathbb C}^{\portsize}$ are the inputs and outputs of the system, respectively. The matrices $A \in {\mathbb C}^{2\FOMsize \times 2\FOMsize}, B \in {\mathbb C}^{2\FOMsize \times \portsize}, C \in {\mathbb C}^{\portsize \times 2\FOMsize}, D \in {\mathbb C}^{\portsize \times \portsize}$ are constant matrices and are called the system matrix, the input matrix, the output matrix, and the direct (input-output) matrix, respectively. The transfer function associated to the system (\ref{eq:FOM}) is given by
	\begin{equation}{\label{eq:TF_FOM}}
		G(s) = C(s I_{2\FOMsize} - A)^{-1}B + D
	\end{equation}
	which is a $p \times p-$matrix-valued rational function, describing the input-output mapping in the frequency domain. 
	
	The \MOR problem can be described as to find a reduced-order \LTI system of the similar form as (\ref{eq:FOM}), i.e., 
	\begin{equation}
		\label{eq:ROM}
		\ROMnota{\Syssymb}\ :\ \ \ \left\{
		\begin{array}{ll}
			\ROMnota{\dot{x}}(t) = \ROMnota{A} \ROMnota{x}(t) + \ROMnota{B}u(t);\ \ROMnota{x}(0) = x_{red,0};\\
			\ROMnota{y}(t) = \ROMnota{C} \ROMnota{x}(t) + \ROMnota{D}u(t),
		\end{array}
		\right.
	\end{equation}
	where $\ROMnota{x}(t) \in {\mathbb C}^{r}$ is the state variable of the reduced system with $\ROMnotaZero{x}$ being the initial condition, dimension $r \in \mathbb{N}$, $r \ll n$, such that for all possible inputs $u(t) \in {\mathbb C}^{p}$ the reduced-order model's output $\ROMnota{y}(t) \in {\mathbb C}^{p}$ approximates the full-order model's output $y(t)$ closely. Here, the dimensions of the reduced matrices are given by $\ROMnota{A} \in {\mathbb C}^{r \times r}, \ROMnota{B} \in {\mathbb C}^{r \times p}, \ROMnota{C} \in {\mathbb C}^{p \times r} $ and $\ROMnota{D}=D \in {\mathbb C}^{p \times p}$. 
	Moreover, the transfer function of the reduced-model
	\begin{equation}{\label{eq:TF_ROM}}
		\ROMnota{G}(s) = \ROMnota{C}(s I_r - \ROMnota{A})^{-1}\ROMnota{B} + \ROMnota{D}
	\end{equation}
	should behave similarly to that of the \FOM.
	
	When the given \FOM $\Syssymb$ is defined via the state-space realizations $A, B,C$, and $D$ as in  (\ref{eq:FOM}), one well-known strategy is to construct the \ROM $\ROMnota{\Syssymb}$, see (\ref{eq:ROM}), via {\em Petrov-Galerkin projection}, see e.g.\ \cite[Chapter 3]{AnBG20}. In this technique, we need to choose two $r$-dimensional subspaces ${\mathcal V}_r\subset {\mathbb C}^{2\FOMsize}$ and ${\mathcal W}_r \subset {\mathbb C}^{2\FOMsize}$ associated with two basis matrices $V, W \in {\mathbb C}^{2\FOMsize \times r}$ such that ${\mathcal V}_r = \text{range}(V)$ and ${\mathcal W}_r = \text{range}(W)$. Then the full-order state is approximated as $x(t) \approx V\ROMnota{x}(t)$ and the residual is constrained to be orthonormal to $W$. We obtain the $r-$dimensional reduced model (\ref{eq:ROM}) with 
	\begin{equation}
		\label{eq:ROM_realization}
		\begin{array}{ll}
			\ROMnota{A} = W^{*}AV,\ \ROMnota{B} = W^{*}B,\ \ROMnota{C}  = CV,\ \ROMnota{D} = D,\ \ROMnotaZero{x} = W^{*}x_0.
		\end{array}
	\end{equation}
	The choice of the subspaces determines the accuracy and effectiveness of the reduced system. Among the various techniques, {\em interpolatory Petrov-Galerkin projection} \cite[Section 3.3]{AnBG20} method is a very effective and, hence, a popular approach for \MOR. This method selects two subspaces ${\mathcal V}_r, {\mathcal W}_r$ to produce the reduced-order model such that its transfer function (\ref{eq:TF_ROM}) interpolates the full-order system's transfer function (\ref{eq:TF_FOM}) at selected interpolation points. Typically, interpolatory projections include Krylov-type subspaces in ${\mathcal V}$ and/or ${\mathcal W}$ to satisfy the interpolation conditions. A more general type subspace, namely rational Krylov subspace \cite[Chapter 11]{BookMOR}, can be used to get a more numerically efficient approach. 
	
	We are particularly interested in {\em rational tangential interpolation}, an extension of the interpolation framework to MIMO systems by \cite{GaVD04}, which refers to interpolation onto some tangent directions. Formally, the rational tangential interpolation problem seeks a reduced system, $\ROMnota{G}(s)$, that interpolates the given full-order system (\ref{eq:TF_FOM}) at selected points $\{s_m\}_{i = 1}^{\fresize} \subset {\mathbb C}$ along selected right tangent directions $\{\mu_m\}_{i = 1}^{\fresize} \subset {\mathbb C}^{p}$, i.e., 
	\begin{equation}{\label{eqn:InterpolationCondition}}
		G(s_m)\mu_m = \ROMnota{G}(s_m)\mu_m\quad \text{for}\quad m = 1,\dots, \fresize.
	\end{equation}
	
	Conditions forcing (\ref{eqn:InterpolationCondition}) to be satisfied by a reduced system of the form (\ref{eq:ROM}) are provided by Theorem \ref{thm:IRKA}, which is an extension of the Proposition 11.7 in \cite{BookMOR} from real systems to complex systems.\footnote{The proof remains unchanged and can be found in \cite{BookMOR}.}
	\begin{theorem}
		\label{thm:IRKA}
		Suppose $G(s) = C(s I_{2\FOMsize} - A)^{-1}B + D$. Given a set of distinct interpolation points $\{s_m\}_{m= 1}^{\fresize} \subset {\mathbb C}$ and right tangent directions $\{\mu_m\}_{m = 1}^{\fresize} \subset {\mathbb C}^{p}$, define $V \in {\mathbb C}^{2\FOMsize \times \fresize}$ as 
		\begin{equation*}
			V= \begin{bmatrix} (s_1 I_{2n} - A)^{-1}B \mu_1, \dots, (s_k I_{2n} - A)^{-1}B \mu_{\fresize}\end{bmatrix}.
		\end{equation*}
		Then for any $W \in {\mathbb C}^{2\FOMsize \times \fresize}$ satisfying $W^{*}V = I_{\fresize}$, the reduced systems $\ROMnota{G}(s) = \ROMnota{C}(s I_{\fresize}- \ROMnota{A})^{-1}\ROMnota{B} + \ROMnota{D}$ defined via (\ref{eq:ROM_realization}) satisfies (\ref{eqn:InterpolationCondition}), provided that $s_m I_{2\FOMsize} - A$ and $s_m I_{\fresize} - \ROMnota{A}$ are invertible. 
	\end{theorem}
	
	\subsection{Review of the port-Hamiltonian representation}
	\label{revPH}
	In this section, we provide a brief introduction to \pH modelling, for details see e.g., \cite{Geoplex}, \cite{jacob2012linear}, \cite{schaft1992}, and \cite{IntroSurvey}. We start by recalling the {\em classical Hamiltonian equations} of motion.
	
	The standard Hamiltonian equations for a mechanical system are defined by 
	\begin{equation}{\label{eqn:ClassicalHamiltonianSystem}}
		\dot{q}_i(t) = \frac{\partial H}{\partial p_i}(q(t),p(t)), \qquad \dot{p}_i(t) = -\frac{\partial H}{\partial q_i}(q(t),p(t)),
	\end{equation}
	for $i = 1,\dots, \FOMsize$, where $H$ is the Hamiltonian of the system, that corresponds for instance to its total energy, $q = 
	\left[ q_1, q_2, \dots, q_{\FOMsize} 
	\right]^{\top}\in {\mathbb C}^{\FOMsize}$ and $p = \left[ p_1, p_2, \dots, p_{\FOMsize} \right]^{\top} \in {\mathbb C}^{\FOMsize}$ are (generalized) position states and momentum states, respectively, and $(q_i,p_i)$-space refers to the generalized phase space of the system. One can easily check that the Hamiltonian is constant along solutions, i.e., $\frac{\mathsf{d}}{\mathsf{d}t}H(q(t),p(t)) = 0$ expressing that the Hamiltonian/energy is conserved within the system.
	
	Now, if we define the state variable 
	$x = \left[ q_1, q_2, \dots, q_{\FOMsize}, p_1, p_2, \dots, p_{\FOMsize} \right]^{\transpose} \in {\mathbb C}^{2\FOMsize}$ 
	for the phase space, then (\ref{eqn:ClassicalHamiltonianSystem}) takes the form
	\begin{align*}
		\dot{x}(t) = {\mathbb J}_{2\FOMsize}\frac{\partial H}{\partial x}(x(t)),
	\end{align*}
	where $\frac{\partial H}{\partial x}(x)$ is the gradient of the Hamiltonian function $H$ and the matrix $\mathbb{J}_{2\FOMsize} \in {\mathbb C}^{2\FOMsize \times 2\FOMsize}$, which is defined by
	\begin{equation}{\label{eq:PossionMatrix}}
		{\mathbb J}_{2\FOMsize} = \begin{bmatrix} 0_{\FOMsize} & I_{\FOMsize}\\-I_{\FOMsize} & 0_{\FOMsize} \end{bmatrix},
	\end{equation}
	is the well-known {\em Poisson matrix}, with $I_{\FOMsize} \in {\mathbb C}^{\FOMsize \times \FOMsize}$ denoting the identity matrix.
	
	Adding external forces, the mechanical system (\ref{eqn:ClassicalHamiltonianSystem}) becomes 
	\begin{align*}
		\dot{q}_i(t) = \frac{\partial H}{\partial p_i}(q(t),p(t)), \qquad \dot{p}_i(t) = -\frac{\partial H}{\partial q_i}(q(t),p(t)) + F_i(t),
	\end{align*}
	where the input $F=\left[ F_1,F_2, \cdots, F_{\FOMsize} \right]^{\transpose}$ is the vector of generalised forces. 
	One can immediately derive the following balance equation: 
	\begin{equation*}
		\frac{dH}{dt}(t) = \frac{\partial H}{\partial q}(q)(t),p(t))^{*}\dot{q}(t) + \frac{\partial H}{\partial p}(q(t),p(t))^{*}\dot{p}(t) = \dot{q}^{*}(t)F(t),
	\end{equation*}
	expressing the balance between internal and external power. This also motivates to define the output of the system as $y(t) = \dot{q}(t)$, and thus the power balance becomes $\frac{\mathsf{d}H}{\mathsf{d}t}(t) = y^{*}(t) F(t)$. With $B=\left[\begin{smallmatrix} 0_n \\ I_n \end{smallmatrix}\right]$, we formulate such system as  
	\begin{equation}
		\label{equ:IRPH_sys}
		\begin{array}{l}
			\dot{x}(t) = J_{2\FOMsize}\frac{\partial H}{\partial x}(x(t)) + B u(t);\quad
			x(0) = x_0;
			\\[1ex]
			y(t) = B^{*}\frac{\partial H}{\partial x}(x(t)),
		\end{array}
	\end{equation}
	where $u=F$. 
	The above \eqref{equ:IRPH_sys} is called a port-Hamiltonian (\pH) system. 
	
	Since we will consider linear systems, we assume that $H$ is quadratic, i.e., $H(x) = \frac{1}{2}x^{*}{{\mathcal H}}x$ with ${\mathcal H} \in {\mathbb C}^{2\FOMsize \times 2\FOMsize}$ symmetric and positive definite. The expression for $H$ gives that $\frac{\partial H}{\partial x}(x(t)) = {\mathcal H} x(t)$, and so (\ref{equ:IRPH_sys}) becomes an \LTI system. Although many \pH systems (without damping) are of the from (\ref{equ:IRPH_sys}), the general form used for linear \pH systems with damping is given by
	\begin{equation}
		\label{eq:LinearPHS}
		\Syssymb_{\text{LpHS}}\ :\ \ \ \left\{
		\begin{array}{l}
			\dot{x}(t) = (J_{2\FOMsize} - R){\mathcal H} x(t) + B u(t);
			\quad
			x(0) = x_0;\\[0.8ex]
			y(t) = B^{*} {\mathcal H} x(t),
		\end{array}
		\right.
	\end{equation}
	with $J_{2\FOMsize}^{*}=-J_{2\FOMsize}$, $R$ representing the damping, satisfying $R=R^{*} \geq 0$.

	Using the properties, the rate of change of Hamiltonian/energy satisfies
	\begin{align*}
		\frac{\mathsf{d}}{\mathsf{d}t}(H(x(t))) =&\ \frac{1}{2} (\langle \dot{x}(t),{\mathcal H} x(t)\rangle + \langle {\mathcal H} x(t),\dot{x}(t)\rangle )\\
		=&\ \mathrm{Re}\langle u(t),y(t) \rangle - \mathrm{Re}\langle R{{\mathcal H}}x(t),{\mathcal H} x(t)\rangle \leq \mathrm{Re}\langle u(t),y(t) \rangle .
	\end{align*}

	\subsection{Model reduction of port-Hamiltonian systems}
	\label{sec:MOR_pH_sys}
	
	Our \pH system as given in (\ref{eq:LinearPHS}) is a special case of the general finite-dimensional \LTI system  (\ref{eq:FOM}).
	If we directly apply the \MOR method presented in Theorem \ref{thm:IRKA}, then the \pH structure will be lost. One way to keep the \pH structure is to use a specially chosen $W$ as in the following theorem.
	\begin{theorem}[\cite{Gugercin2012IRKA}]{\label{thm:IRKA-pH}}
		Consider the linear \pH system of equation (\ref{eq:LinearPHS}). Given a set of distinct interpolation points $\{s_m\}_{m= 1}^{\fresize} \subset {\mathbb C}$ and right tangent directions $\{\mu_m\}_{m = 1}^{\fresize} \subset {\mathbb C}^{p}$, define $V \in {\mathbb C}^{2\FOMsize \times \fresize}$ as 
		\begin{equation}
			\label{equ:Kry_basis_pH}
			V= \begin{bmatrix} (s_1 I_{2\FOMsize} - (J_{2\FOMsize}-R){{\mathcal H}})^{-1}B\mu_1, \dots, (s_M I_{2\FOMsize} - (J_{2\FOMsize}-R){{\mathcal H}})^{-1}B\mu_M\end{bmatrix}.
		\end{equation}
		Then for $W \in {\mathbb C}^{2\FOMsize \times \fresize}$ given by $W:= {{\mathcal H}}V(V^{*}{{\mathcal H}} V)^{-1}$, the reduced \pH system  satisfying (\ref{eqn:InterpolationCondition}) is given as
		\begin{equation}
			\label{equ:ROM_sys_mats}
			\begin{array}{rcl}
				\ROMnota{J} &=&W^{*} J_{2\FOMsize} W, 
				\quad 
				\ROMnota{R} = W^{*} R W, 
				\quad
				\ROMnota{{\mathcal H}} = V^{*} {\mathcal H} V,
				\\\
				\ROMnota{B} &=& W^{*} B,
				\quad
				\ROMnotaZero{x} = W^{*}x_0.
			\end{array}
		\end{equation}
	\end{theorem}
	\begin{proof}
		By the choice of $W$ we have that ${\mathcal H} V= W {\mathcal H}$ which provides the equality $\ROMnota{A} := (\ROMnota{J} - \ROMnota{R}) \ROMnota{{\mathcal H}} = W^{*} (J - R){\mathcal H} V = W^{*}AV$. Now
		the proof follows from Theorem \ref{thm:IRKA}. 
	\end{proof}
	
	We denote the \MOR method based on Theorem \ref{thm:IRKA-pH} as \emph{pH-Arnoldi} method.
	As we have seen in the previous subsection the $J_{2\FOMsize}$ matrix is very often of the special form (\ref{eq:PossionMatrix}). The reduced $\ROMnota{J}$ of (\ref{equ:ROM_sys_mats}) will loose this form. 
	Although when $\fresize=2\ROMsize$, a basis transformation on the reduced system can bring this to the required form ${\mathbb J}_{2\ROMsize}$, it would be nice to have transformations $V$ and $W$ which does this in one step. 
	Furthermore, in many linear \pH systems the Hamiltonian $H$ is the energy and consists of the two parts, kinetic and potential energy. The kinetic energy only depends on the momentum and the potential energy only on the position. Moreover, the dissipation term $R$ often only depends on the velocity/momentum. A canonical example is the mass spring damper (\MSD) system. 
	For electrical circuits, if the system is modelled by choosing the magnetic fluxes of inductors and electric charges of capacitors as the state, then a similar equation can be found where the dissipation terms depends on both the magnetic fluxes and the electric charges. This leads to a different block structure in \eqref{eq:pH-block_FOM}, as $R_1\neq 0_{\FOMsize}$ and $\tilde{J}$ does not need to equal the identity matrix anymore. For this, we need a novel class of \MOR methods, which we develop in the following section. 
	
	\section{Symplectic model reduction}
	\label{sec:symplecticMOR}
	
	This section provides a structure-preserving \MOR technique namely {\em symplectic} \MOR which can be applied for a linear \pH system $\Syssymb_{\mathsf{LpHS}}$ in (\ref{eq:LinearPHS}). Symplectic \MOR was originally introduced for autonomous Hamiltonian systems, see \cite{Symplectic, AfkH17}. This work can be viewed as an extension of the work of \cite{Symplectic} to \pH system with noncanonical Hamiltonian structure, i.e., $J_{2\FOMsize} \neq \mathbb{J}_{2\FOMsize}$ and non-zero dissipation. We start by recalling some essentials of symplectic geometry on vector spaces in Section \ref{sec:symplecticMOR.1} that are necessary to understand the methods used. Further, in Section \ref{sec:symp_MOR_pH} we state some definitions and propositions that are necessary to prove the main result in Theorem \ref{th:pH_with_Dissipation}. Since we interpolate at frequency points, which are complex, we need to introduce symplectic forms on complex vector spaces. 
	
	\subsection{Symplectic geometry on complex vector spaces} \label{sec:symplecticMOR.1}
	We introduce the essential definitions and theorems of symplectic geometry in this section, which are needed for the remainder of the paper. The content in the following part can be found in standard textbooks, as e.g., \cite{marsden1999introduction}.
	\begin{definition} 
		\label{D3.1}
		Let $\V$ be a vector space over ${\mathbb C}$ of even dimension. A {\em symplectic form}\/ on $\V$ is a mapping of the form $\Omega : \V \times \V \rightarrow {\mathbb C}$ being linear in the first component, anti-linear in the second, and satisfying
		\begin{enumerate}
			\item {\em Skew-symmetry}: $\Omega(v_1,v_2) = -\overline{\Omega(v_2,v_1)}$ for all $v_1,v_2 \in \V$;
			\item {\em Non-degenerate}: $\Omega(v_1,v_2) = 0$ for all  $v_1\in\V$ if and only if  $v_2 = 0$.
		\end{enumerate}
	\end{definition}
	
	The pair $(\V,\Omega)$ is called a {\em symplectic vector space} with respect to the symplectic form $\Omega$. On finite-dimensional vector spaces any full rank skew-symmetric matrix induces a symplectic form.
	\begin{proposition}
		{\label{prop:SympForm_J_2n}}
		Let $J_{2\FOMsize} \in {\mathbb C}^{2\FOMsize\times 2\FOMsize}$ be a full-rank skew-symmetric structure matrix, i.e., $J_{2\FOMsize}^{*} = - J_{2\FOMsize}$. Then $\Omega_{J_{2\FOMsize}}$ defined as 
		\begin{equation}
			\label{eqn:SympForm_J_2n}
			\Omega_{J_{2\FOMsize}}(u,v) = \langle J_{2\FOMsize} u,v\rangle=v^{*} J_{2\FOMsize} u, \qquad u,v \in {\mathbb C}^{2\FOMsize}
		\end{equation}
		is a symplectic form on ${\mathbb C}^{2\FOMsize}$.
	\end{proposition}
	\begin{proof} 
		The (anti)linearity of $\Omega_{J_{2\FOMsize}}$ follows directly from the (anti)linearity of the inner product. The antisymmetry is a direct consequence of the skew-symmetry of $J_{2\FOMsize}$.
		So it remains to show that the form is non-degenerate. 
		
		If $\Omega_{J_{2\FOMsize}}(u,v)=0$ for all $v \in {\mathbb C}^{2\FOMsize}$, then $J_{2\FOMsize} u$ must equal the zero vector. Since $J_{2\FOMsize}$ is full rank, this implies that $u=0$, and thus the form is non-degenerated. 
	\end{proof}
	
	Next we introduce symplectic mappings, i.e., mappings that keep the symplectic structures.
	\begin{definition}
		\label{defn:SymplMap_J_2n}
		Let $({\mathbb C}^{2\FOMsize}, \Omega_{J_{2\FOMsize}})$ and $({\mathbb C}^{2\ROMsize}$,$\Omega_{J_{2\ROMsize}})$ be two symplectic vector spaces with $\ROMsize,\FOMsize \in \mathbb{N}$ and $\ROMsize\le\FOMsize$ where the symplectic form $\Omega_{J_{2\FOMsize}}$ and $\Omega_{J_{2\ROMsize}}$ are defined as in (\ref{eqn:SympForm_J_2n}) and $Q: {\mathbb C}^{2\ROMsize} \rightarrow {\mathbb C}^{2\FOMsize},\ y \mapsto Q y,\ Q \in {\mathbb C}^{2\FOMsize \times 2\ROMsize}$ be a linear mapping. The matrix $Q$ is called a $(J_{2\FOMsize},J_{2\ROMsize})$-{\em symplectic} if 
		\begin{equation}
			\label{eqn:SymplMap_J_2n}
			Q^{*}J_{2\FOMsize}Q = J_{2\ROMsize}.
		\end{equation}
	\end{definition}
	
	Given a $(J_{2\FOMsize},J_{2\ROMsize})$-symplectic matrix $Q$, it is easy to see that 
	\[
	\Omega_{J_{2\FOMsize}}(Q u,Q v) = \langle  J_{2\FOMsize}Q u, Q v \rangle = \langle  Q^* J_{2\FOMsize} Q u,v \rangle =\langle J_{2\ROMsize}u , v \rangle = \Omega_{J_{2\ROMsize}}(u,v),
	\]
	and thus it preserves the symplectic structure. An important object in terms of symplectic \MOR related to a symplectic mapping is the following symplectic inverse. 
	
	\begin{definition}
		\label{D:3.4}
		For a $(J_{2\FOMsize},J_{2\ROMsize})$-symplectic matrix $Q \in {\mathbb C}^{2\FOMsize \times 2\ROMsize}$, the {\em symplectic inverse}, denoted by $Q^{\Linv}$, is defined by
		\begin{equation}
			\label{equ:mat_Q_Linv}
			Q^{\Linv} := J_{2\ROMsize}^{-1}Q^{*}J_{2\FOMsize}\ \in {\mathbb C}^{2\ROMsize \times 2\FOMsize}.
		\end{equation}
	\end{definition}

	Among others the symplectic inverse is a left inverse, as is shown next. This lemma is an extension of \cite[Lemma 3.3 and 3.4]{Symplectic} to the non-canonical and complex domain.
	\begin{lemma}
		\label{lem:3.5}
		Let $Q \in {\mathbb C}^{2\FOMsize \times 2\ROMsize}$ be a $(J_{2\FOMsize},J_{2\ROMsize})$-symplectic matrix and let $Q^{\Linv}$ be the symplectic inverse from Definition \ref{D:3.4}. 
		The following holds
		\begin{enumerate}
			\item $Q^{\Linv}Q = I_{2k}$; 
			\item $Q^{\Linv}J_{2\FOMsize}^{-1} =J_{2\ROMsize}^{-1}Q^{*}$; 
			\item $(Q^{\Linv})^{*}$ is $(J_{2\FOMsize}^{-1},J_{2\ROMsize}^{-1})$-symplectic; 
			\item If $v \in \text{ran}(Q)$, then $v = QQ^{\Linv}v$.
		\end{enumerate}
	\end{lemma}
	\begin{proof}
		{\em 1.}\ Using the expression of $Q^{\Linv}$ of (\ref{equ:mat_Q_Linv}) we see that
		\[
		Q^{\Linv}Q = J_{2\ROMsize}^{-1} Q^{*}J_{2\FOMsize}Q \overset{(\ref{eqn:SymplMap_J_2n})}{=} J_{2\ROMsize}^{-1} J_{2\ROMsize} =I_{2k}.
		\]
		{\em 2.}\ Similarly as above we have
		\[
		Q^{\Linv} J_{2\FOMsize}^{-1} = J_{2\ROMsize}^{-1}Q^{*}J_{2\FOMsize} J_{2\FOMsize}^{-1} = J_{2\ROMsize}^{-1}Q^{*} .
		\]
		{\em 3.}\ To show that $(Q^{\Linv})^{*}$ is $(J_{2\FOMsize}^{-1},J_{2\ROMsize}^{-1})$-symplectic, see \eqref{eqn:SymplMap_J_2n}, we have to show that
		\[
		(Q^{\Linv}) J_{2\FOMsize}^{-1}(Q^{\Linv})^{*} = J_{2\ROMsize}^{-1}.
		\]
		Note that $J_{2\FOMsize}^{-1}$, and $J_{2\ROMsize}^{-1}$ are full rank, skew-symmetric matrices. 
		Using the expression on $Q^{\Linv}$, we find
		\begin{align*}
			Q^{\Linv} J_{2\FOMsize}^{-1}(Q^{\Linv})^{*} =&\ J_{2\ROMsize}^{-1}Q^{*}J_{2\FOMsize} J_{2\FOMsize}^{-1}
			(J_{2\ROMsize}^{-1}Q^{*}J_{2\FOMsize})^{*} = 
			(-1)^2 J_{2\ROMsize}^{-1}Q^{*}J_{2\FOMsize} J_{2\FOMsize}^{-1}
			J_{2\FOMsize}^{*}Q J_{2\ROMsize}^{-*}\\
			=&\
			J_{2\ROMsize}^{-1}Q^{*}  J_{2\FOMsize}^{*}Q J_{2\ROMsize}^{-*} = J_{2\ROMsize}^{-1} J_{2\ROMsize}^{*} J_{2\ROMsize}^{-*} = J_{2\ROMsize}^{-1}.
		\end{align*}
		{\em 4.}\ If $v \in \mathrm{ran}(Q)$, then there exists a $w \in {\mathbb C}^{2\FOMsize}$ such that $v = Q w$. By item 1.\ this implies
		\[
		Q^{\Linv}v = Q^{\Linv}Q w = w 
		\]
		and so $v = Q w = QQ^{\Linv}v$. 
	\end{proof}
	For the canonical symplectic matrices $\mathbb{J}_{2\FOMsize}$ and $\mathbb{J}_{2\ROMsize}$, the above expressions become simpler, since 
	\begin{equation*}
		\mathbb{J}_{2\FOMsize}^{-1} = -\mathbb{J}_{2\FOMsize} = \mathbb{J}_{2\FOMsize}^{\mathrm T}=\mathbb{J}_{2\FOMsize}^*.
	\end{equation*}
	In that case, we don't call $Q$ a $({\mathbb J}_{2\FOMsize}, {\mathbb J}_{2\ROMsize})$-symplectic matrix, but just $\mathbb{J}_{2\FOMsize}$-{\em symplectic}.
	
	\subsection{Symplectic model reduction of port-Hamiltonian systems} \label{sec:symp_MOR_pH}
	
	In this section, we derive symplectic \MOR techniques preserving the port-Hamiltonian structure. We consider two cases: 
	\begin{enumerate}
		\item Without dissipation, i.e., $R = 0_{2\FOMsize}$, in Section \ref{Subsec:wo_dissipation};
		\item With dissipation, i.e., $R \neq 0_{2\FOMsize}$, in Section \ref{Subsec:w_dissipation}.
	\end{enumerate}
	Further, in Section \ref{sec:construction of matrices}, we detail the construction of matrices needed in the \MOR method described in Section \ref{Subsec:w_dissipation}. We start with the case without dissipation.
	
	\subsubsection{Port-Hamiltonian system without dissipation}\label{Subsec:wo_dissipation}
	
	Consider the \pH system (\ref{eq:LinearPHS}) on the symplectic vector space (${\mathbb C}^{2\FOMsize}, \Omega_{J_{2\FOMsize}}$) with $R = 0_{2\FOMsize}$. In the following theorem we give conditions under which we can find a reduced order system (\ref{eq:MOR_LinearPHS}), without dissipation, such that the reduced and original transfer functions are equal at prescribed frequencies. 
	
	\begin{theorem}
		{\label{thm:PHS_NoDissipation}}
		Consider the \pH system without dissipation, i.e., $\Syssymb_{LpHS}$ of (\ref{eq:LinearPHS}) with $R=0_{2\FOMsize}$.
		Let $Q \in {\mathbb C}^{2\FOMsize \times 2\ROMsize}$ be a $(J_{2\FOMsize}^{-1},J_{2\ROMsize}^{-1})$-symplectic matrix, and let $Q^{\Linv} \in {\mathbb C}^{2\ROMsize \times 2\FOMsize}$ be its  symplectic inverse, where $J_{2\ROMsize}$ is the skew-symmetric matrix of the reduced system. 
		
		Finally, let $\mu_1, \cdots, \mu_M \in {\mathbb C}^p$ and $s_1, \ldots,s_M \in {\mathbb C}$ be given.
		
		If ${\mathrm{ran}}(B) \subseteq {\mathrm {ran}}(Q)$ and $(s_m I_{2\FOMsize}- J_{2\FOMsize}{\mathcal H})^{-1} B\mu_m \in {\mathrm {ran}}(Q) $, $m=1,\ldots,\fresize$, then the reduced \pH system (\ref{eq:MOR_LinearPHS}) with $\ROMnota{B} = Q^{\Linv}B$ and $\ROMnota{{\mathcal H}} = Q^{*}{\mathcal H}Q$ satisfies 
		$\ROMnota{G}(s_m)\mu_m = \ROMnota{G}(s_m)\mu_m$, $m=1,\cdots, \fresize,$ i.e., the reduced \pH system equals the full \pH system at the given frequencies. 
		Moreover, $\lim_{s\rightarrow\infty}sG(s) = \lim_{s\rightarrow\infty}sG_r(s)$ holds.
	\end{theorem}
	\begin{proof}
		Let us denote $(s_m I_{2\FOMsize}- J_{2\FOMsize}{\mathcal H})^{-1} B\mu_m$ by $\eta_{m}$, then $G(s_m)\mu_m= B^{*}{\mathcal H} \eta_m$. Since $\eta_m \in {\mathrm{ran}}(Q)$, we have by Lemma \ref{lem:3.5}.4 that $\eta_m = QQ^{\Linv}\eta_m$.
		Now,
		\begin{align}
			\label{eq:3.6}
			s_m\eta_m - J_{2\FOMsize}{\mathcal H} \eta_m - B\mu_m =&\ 0 \Rightarrow\\
			\nonumber
			s_m Q^{\Linv}\eta_m - Q^{\Linv}J_{2\FOMsize}{{\mathcal H}}\eta_m - Q^{\Linv}B\mu_m =&\ 0 \Rightarrow\\
			\nonumber
			s_m Q^{\Linv}\eta_m - J_{2\ROMsize}Q^{*}{\mathcal H}QQ^{\Linv}\eta_m - Q^{\Linv}B\mu_m =&\ 0 \Rightarrow\\
			\label{eq:3.7}
			s_m z_m- J_{2\ROMsize}\ROMnota{{\mathcal H}}z_m - \ROMnota{B}\mu_m =&\ 0,
		\end{align}
		with $z_m = Q^{\Linv}\eta_m$, $\ROMnota{{\mathcal H}} = Q^{*}{\mathcal H}Q$, and $\ROMnota{B} = Q^{\Linv}B$, and where the third implication follows from Lemma \ref{lem:3.5}.3. 
		Since $\mathrm{ran}(B) \subseteq \mathrm{ran}(Q)$, we have by Lemma \ref{lem:3.5}.4 that $B\mu = QQ^{\Linv}B\mu$ for all $\mu$. Equivalently, $B^*= B^{*}(Q^{\Linv})^{*}Q^{*}$. Therefore, 
		\begin{equation}
			\label{eq:3.9}
			G(s_m)\mu_m = B^{*}{{\mathcal H}}\eta_m = B^{*}(Q^{\Linv})^{*}Q^{*}{\mathcal H}QQ^{\Linv}\eta_m = \ROMnota{B}^{*}\ROMnota{{\mathcal H}}z_m,
		\end{equation}
		where we used once more that $\eta_m = QQ^{\Linv}\eta_m$. 
		From (\ref{eq:3.6}), (\ref{eq:3.7}), and (\ref{eq:3.9}), it follows that $\ROMnota{G}(s_m)\mu_m = G(s_m)\mu_m$.
		It is easy to check that $\ROMnota{{\mathcal H}}$ is symmetric and positive semi-definite, i.e., $\ROMnota{{\mathcal H}}^{*} = \ROMnota{{\mathcal H}} \geq 0_{2\FOMsize}$. Hence, the reduced system preserves the \pH structure. Furthermore, since $\lim_{s\rightarrow\infty} s(s_m I_{2\FOMsize}- J_{2\FOMsize}{\mathcal H})^{-1} B = B$ and $\mathrm{ran}(B) \subseteq \mathrm{ran}(Q)$, we have $\lim_{s\rightarrow\infty}sG(s) = \lim_{s\rightarrow\infty}s\ROMnota{G}(s)$.
	\end{proof}
	
	Given the conditions in the above theorem we see that $k \geq \fresize+p$.
	
	
	\subsubsection{Port-Hamiltonian system with dissipation}\label{Subsec:w_dissipation}

	In the dissipation-free case, we saw that $Q^{-L} J_{2\FOMsize}{\mathcal H}Q= J_{2\ROMsize} \ROMnota{{\mathcal H}}$. However, when $R\neq 0$, then  $Q^{-L} R{\mathcal H} Q\neq  \ROMnota{R}\ROMnota{{\mathcal H}}$. Thus in this situation, we have to adapt our construction. 
	\begin{theorem}
		\label{th:pH_with_Dissipation}
		Consider the full-order \pH system with dissipation \eqref{eq:LinearPHS}. Take 
		\begin{equation}
			\label{eq:QE}
			Q_{\mathsf{E}} := 
			\begin{bmatrix} 
				Q_1 & 0_{2\FOMsize\times 2\ROMsize}\\ 
				0_{2\FOMsize\times 2\ROMsize} & Q_2 
			\end{bmatrix}
			\in \mathbb{C}^{4\FOMsize \times 4\ROMsize},
		\end{equation}
		where $0_{2\FOMsize\times 2\ROMsize}\in {\mathbb C}^{2\FOMsize\times 2\ROMsize}$ denotes the zero matrix, $Q_2 \in \mathbb{C}^{2\FOMsize \times 2\ROMsize}$ with $k<n$ is $(J_{2\FOMsize},J_{2\ROMsize})$-symplectic, and $Q_1^{*}$ is a left inverse of $Q_2$. Here $J_{2\ROMsize}$ is the skew-symmetric matrix of the reduced \pH system.
		
		Let $\mu_1, \ldots, \mu_{\fresize} \in {\mathbb C}^p$, $s_1, \ldots,s_{\fresize} \in {\mathbb C}$ be given.
		If $(s_m I_{2\FOMsize}- (J_{2\FOMsize}-R){\mathcal H})^{-1} B \mu_m \in {\mathrm {ran}}(Q_1) $, $m=1,\ldots,\fresize$ and ${\mathrm {ran}}({\mathcal H} Q_1) \subseteq {\mathrm {ran}}(Q_2)$, 
		then the (reduced-order) \pH system with $\ROMnota{B}=Q_2^{*} B$, $\ROMnota{\mathcal H}=Q_1^{*} {\mathcal H}Q_1$, and $\ROMnota{R}= Q_2^{*} R Q_2$ satisfies
		$G(s_m)\mu_m = \ROMnota{G}(s_m)\mu_m$ for $m=1,\ldots,\fresize$.
	\end{theorem}
	\begin{proof}
		Applying the definitions for the reduced model, we find that its transfer function is given by
		\begin{align}
			\nonumber
			\ROMnota{G}(s_m)\mu_m
			&=
			\ROMnota{B}^{*}\ROMnota{\mathcal H}(s_m I_{2\ROMsize}- (J_{2\ROMsize}-\ROMnota{R})\ROMnota{\mathcal H})^{-1}\ROMnota{B} \mu_m\\
			\nonumber
			&=
			B^{*}Q_2 Q_1^*{\mathcal H}Q_1
			(s_m I_{2\ROMsize} - (Q_2^*J_{2\FOMsize}Q_2 - Q_2RQ_2)Q_1^*{\mathcal H} Q_1)^{-1}
			Q_2^{*}B \mu_m\\
			&=
			\label{eq:29}
			B^{*}Q_2 Q_1^*{\mathcal H}Q_1
			(s_m I_{2\ROMsize} - Q_2^*(J_{2\FOMsize}-R)Q_2Q_1^*{\mathcal H}Q_1)^{-1}
			Q_2^{*}B \mu_m.
		\end{align}
		Since $(s_m I_{2\FOMsize}-(J_{2\FOMsize}-R){\mathcal H})^{-1}B \mu_m$ lies by assumption in the range of $Q_1$, there exists $\xi_m \in {\mathbb C}^{2\FOMsize}$ such that  $Q_1\xi_m  = (s_m I_{2\FOMsize}-(J_{2\FOMsize}-R){\mathcal H})^{-1}B \mu_m$. This implies that
		\[
		B \mu_m = (s_m I_{2\FOMsize}-(J_{2\FOMsize}-R){\mathcal H})Q_1\xi_m .
		\]
		As $Q_1^*Q_2=I$, and thus $Q_2^*Q_1=I$, we find that
		\[
		Q_2^*B \mu_m = (s_m I_{2\ROMsize}-Q_2^* (J_{2\FOMsize}-R){\mathcal H}Q_1)\xi_m. 
		\]
		Applying this to the equality in (\ref{eq:29}) yields
		\begin{align*}
			\ROMnota{G}(s_m)\mu_m =&\ B^{*}Q_2 Q_1^*{\mathcal H}Q_1
			(s_m I_{2k} - Q_2^*(J_{2\FOMsize}-R)Q_2Q_1^*{\mathcal H}Q_1)^{-1}\\
			&\
			(s_m I_{2k}-Q_2^* (J_{2\FOMsize}-R){\mathcal H}Q_1)\xi_m.
		\end{align*}
		The inclusion ${\mathrm {ran}}({\mathcal H} Q_1) \subseteq {\mathrm {ran}}(Q_2)$ combined with the fact that $Q_1^*$ is a left inverse of $Q_2$ implies that ${\mathcal H} Q_1 = Q_2 Q_1^* {\mathcal H} Q_1$. With this the above equality can be simplified to
		\begin{align*}
			\ROMnota{G}(s_m)\mu_m =&\ B^{*}{\mathcal H}Q_1
			(s_m I_{2k} - Q_2^*(J_{2\FOMsize}-R){\mathcal H}Q_1)^{-1}\\
			&\
			(s_m I_{2k}-Q_2^* (J_{2\FOMsize}-R){\mathcal H}Q_1)\xi_m\\
			=&\  B^{*}{\mathcal H}Q_1 \xi_m\\
			=&\
			B^{*}{\mathcal H}(s_m I_{2\FOMsize}-(J_{2\FOMsize}-R){\mathcal H})^{-1}B \mu_m 
			=
			G(s_m)\mu_m.
		\end{align*}
		Hence we have proved the theorem.
	\end{proof}
	
	The result as presented in the theorem above relies heavily on the existence of two matrices $Q_1$ and $Q_2$ satisfying certain properties. In the following subsection, we show that these matrices can be constructed and we provide an algorithm for it. 
	
	\subsubsection{Construction of matrices for symplectic-\pH with dissipation}
	\label{sec:construction of matrices}
	
	In Theorem \ref{th:pH_with_Dissipation} we showed that the reduced \pH system can be constructed once we have constructed the matrix $Q_E$ (\ref{eq:QE}) with the mentioned properties. In this section, we give a constructive instruction how this matrix can be realized computationally. In the remainder of this section, we show that with this $Q_E$ the block diagonal structure of the \FOM is maintained in the \ROM.
	
	To construct $Q_E$ satisfying the properties as needed in Theorem \ref{th:pH_with_Dissipation}, we use $Q_1$ and $Q_2$ with the following structure 
	\begin{equation}
		\label{eq:mat_Q_structure}
		Q_1 = 
		\left[
		\begin{array}{cc}
			\Phi_1 & 0_{\FOMsize\times \ROMsize} \\
			0_{\FOMsize\times \ROMsize} & \Psi_1
		\end{array}
		\right]
		,\quad
		Q_2 = 
		\left[
		\begin{array}{cc}
			\Phi_2 & 0_{\FOMsize\times \ROMsize} \\
			0_{\FOMsize\times \ROMsize} & \Psi_2
		\end{array}
		\right],
	\end{equation}
	where $\Phi_1,\Psi_1,\Phi_2,\Psi_2\in{\mathbb C}^{\FOMsize\times \ROMsize}$.
	Assuming that the $J$ matrices have the following structure
	\begin{equation*}
		J_{2\FOMsize} = 
		\left[
		\begin{array}{cc}
			0_{\FOMsize}& \hat{J}_{\FOMsize}\\
			-\hat{J}_{\FOMsize}^{*} &0_{\FOMsize}
		\end{array}
		\right],
		\quad
		J_{2\ROMsize} = 
		\left[
		\begin{array}{cc}
			0_{\ROMsize}& \hat{J}_{\ROMsize}\\
			-\hat{J}_{\ROMsize}^{*} &0_{\ROMsize}
		\end{array}
		\right],
	\end{equation*}
	the conditions in Theorem \ref{th:pH_with_Dissipation} imply that $\Phi_1,\Psi_1,\Phi_2,$ and $\Psi_2$ in (\ref{eq:mat_Q_structure}) satisfy 
	\begin{equation}
		\label{equ:condition_for_Phi_and_Psi}
		\Phi_1^{*}\Phi_2 = I_{\ROMsize},\quad
		\Psi_1^{*}\Psi_2 = I_{\ROMsize},\quad
		\Phi_2^{*}\hat{J}_{\FOMsize}\Psi_2 = \hat{J}_{\ROMsize}.
	\end{equation}
	
	For a given interpolation data $\mu_1, \ldots, \mu_{\fresize} \in {\mathbb C}^p$, $s_1, \ldots,s_{\fresize} \in {\mathbb C}$, we construct two matrices $V_1, V_2\in{\mathbb C}^{2\FOMsize\times \ROMsize}$ such that
	\begin{equation}
		\label{eq:30}
		\begin{array}{l}
			(s_m I_{2\FOMsize} - (J_{2\FOMsize}-R){\mathcal H})^{-1} B \mu_m \in {\mathrm {ran}}(V_1),
			\
			m=1,\ldots, \fresize,\\
			\mbox{ and }\quad
			{\mathrm {ran}}({\mathcal H} V_1) \subseteq {\mathrm {ran}}(V_2).
		\end{array}
	\end{equation}
	Moreover, we partition $V_1$ and $V_2$ as 
	\begin{equation}
		\label{eq:mat_V_partition}
		V_1 = \left[
		\begin{array}{c}
			V_{11}\\
			V_{12}
		\end{array}
		\right],
		\quad
		V_2= \left[
		\begin{array}{c}
			V_{21}\\
			V_{22}
		\end{array}
		\right]
	\end{equation}
	with $V_{11}, V_{12}, V_{21},V_{22}\in{\mathbb C}^{\FOMsize\times \ROMsize}$. Using this notation, we see from (\ref{eq:mat_Q_structure}) and (\ref{eq:30}) that the range conditions of Theorem \ref{th:pH_with_Dissipation} are satisfied when ran$(V_{11}) \subseteq$ ran$(\Phi_1)$, ran$(V_{12}) \subseteq$ ran$(\Psi_1)$, ran$(V_{21}) \subseteq$ ran$(\Phi_2)$, and ran$(V_{22}) \subseteq$ ran$(\Psi_2)$. 
	To ensure that these conditions and the conditions given by \eqref{equ:condition_for_Phi_and_Psi} hold, we construct $\Phi_1, \Psi_1,\Phi_1$, and $\Psi_2$ by
	\begin{equation}
		\label{eq:const_phi_psi}
		\begin{array}{rclrcl}
			\Psi_2 &=&\ {\rm orth}(V_{22}),&
			\Psi_1 &=& V_{12}(\Psi_2^{*} V_{12})^{-1},
			\\
			\Phi_2 &=&\ V_{21} \left( \Psi_2^{*}\hat{J}_{\FOMsize}^{*}V_{21}\right)^{-1}\hat{J}_{\ROMsize}^{*},&\
			\Phi_1 &=& V_{11}(\Phi_2^{*}V_{11})^{-1},
		\end{array}
	\end{equation}
	where the operator $`` {\rm orth}"$ returns an orthogonal basis of the same dimension, based on the singular value decomposition. Additionally, we assume the existence of the inverses in \eqref{eq:const_phi_psi}. 
	Since the above construction is explicit, it can be implemented straight-forwardly by the pseudo code provided in Algorithm \ref{alg:sym-pH}.
	\begin{algorithm}
		\caption{Symplectic \MOR for \pH systems (symplectic-\pH)}
		\label{alg:sym-pH}
		\renewcommand{\algorithmicrequire}{\textbf{Input:}}
		\renewcommand{\algorithmicensure}{\textbf{Output:}}
		\begin{algorithmic}
			\Require \pH system \eqref{eq:pH-block_FOM}, reduced-order $J_{2\ROMsize}$, $\mu_1, \ldots, \mu_{\fresize} \in {\mathbb C}^p$, $s_1, \ldots,s_{\fresize} \in {\mathbb C}$.
			\Ensure Reduced-order \pH system of form of \eqref{equ:sympl_pH_ROM}.
			\State
			Construct $V_1$ by \eqref{equ:Kry_basis_pH}.
			\State
			Denote $V_2$ by $V_2 \coloneqq \mathcal{H}V_1$.
			\State
			Build $V_{11},V_{12},V_{21},V_{22}$ based on the partition given by \eqref{eq:mat_V_partition}.
			\State
			Construct $\Psi_1$, $\Psi_2$, $\Phi_1$, $\Phi_2$ by \eqref{eq:const_phi_psi}.
			\State
			Build $Q_1$ and $Q_2$ by \eqref{eq:mat_Q_structure}
			\State
			Apply Petrov-Galerkin with projection bases $Q_1$ and $Q_2$ to \eqref{eq:pH-block_FOM}.
		\end{algorithmic}
	\end{algorithm}
	
	Since the $Q_1$ and $Q_2$ are block diagonal, the special block structure of the \FOM given by \eqref{eq:pH-block_FOM} is preserved in the \ROM. Thus, with matrices $Q_1$ and $Q_2$ of (\ref{eq:mat_Q_structure}), we arrive at the reduced-order \pH system 
	\begin{equation}
		{\label{equ:sympl_pH_ROM}}
		\Syssymb_{\text{sym-pH-ROM}}\ 
		\colon
		\left\{
		\begin{array}{l}
			\ROMnota{\dot{x}}(t) = (\ROMnota{J}-\ROMnota{R})\ROMnota{{\mathcal H}} \ROMnota{x}(t) + \ROMnota{B} u(t);\
			\ROMnota{x}(0)=x_{red,0};\\
			\ROMnota{y}(t) = \ROMnota{B}^{\top}\ROMnota{{\mathcal H}} \ROMnota{x}(t),
		\end{array}
		\right.
	\end{equation}
	where
	\begin{equation}
		\label{eq:34}
		\begin{array}{c}
			\ROMnota{J} = J_{2\ROMsize} ,\quad
			\ROMnota{R} = \left[
			\begin{array}{cc}
				\Phi_2^{*}R_1\Phi_2 & 0_{\ROMsize}\\
				0_{\ROMsize}        &\Psi_2^{*}R_2\Psi_2
			\end{array}
			\right]
			,\quad
			\ROMnota{B} = 
			\left[
			\begin{array}{c}
				\Phi_2^{*}B_1\\
				\Psi_2^{*}B_2
			\end{array}
			\right],
			\\
			\ROMnota{{\mathcal H}} = \left[
			\begin{array}{cc}
				\Phi_1^{*}{\mathcal H}_1\Phi_1&0_{\ROMsize}\\
				0_{\ROMsize}                &\Psi_1^{*}{\mathcal H}_2\Psi_1
			\end{array}
			\right]
			,
			\quad
			\ROMnotaZero{x} 
			=
			\left[
			\begin{array}{cc}
				\Phi_2^* & 0_{\ROMsize\times\FOMsize} \\
				0_{\ROMsize\times \FOMsize}& \Psi_2^*
			\end{array}
			\right]
			x_0
			.
		\end{array}
	\end{equation}
	
	\section{Mass spring damper systems}
	\label{sec:MSD}
	
	As an example of the theory derived in the previous section, we study the mass spring damper system (\MSD). That is we 
	consider the second order system of the form
	
	\begin{equation}
		\label{eq:MSD_sys}
		\begin{array}{l}
			M \ddot{z}(t) + D \dot{z}(t) + K z(t) = \SoBmat u(t);\quad
			z(0)=z_0;\\
			y(t) = \SoBmat^{*} \dot{z}(t), 
		\end{array}
	\end{equation}
	with $z(t) \in {\mathbb C}^{\FOMsize}$, $M=M^{*}  >0$, $K=K^{*}  >0$, and $D=D^{*} \geq 0.$\footnote{Note that here the $D$ denotes the damping, and not the feedthrough operator of (\ref{eq:FOM}).} It is well-know that this system can be written as a \pH system by introducing the state (position and momentum)
	\[
	x(t) = \left[ \begin{array}{c} z(t) \\ M \dot{z}(t) \end{array} \right].
	\]
	The corresponding \pH system equals 
	\begin{equation}
		\label{eq:SO_pH_FO_form}
		\begin{array}{rcl}
			\dot{x}(t)
			&=&\ \left( 
			\left[\begin{array}{cc} 
				0_{\FOMsize} & I_{\FOMsize} \\
				-I_{\FOMsize} &0_{\FOMsize}  
			\end{array} \right] 
			-  \left[\begin{array}{cc} 
				0_{\FOMsize} & 0_{\FOMsize}  \\
				0_{\FOMsize} & D \end{array} \right] \right)  
			\left[ \begin{array}{cc}
				K & 0_{\FOMsize} \\ 
				0_{\FOMsize} & M^{-1} 
			\end{array} \right]
			x(t)
			+ \left[ \begin{array}{c} 
				0_{\FOMsize\times \portsize} \\ \SoBmat 
			\end{array} \right] u(t) ,
			\\
			y(t) &=&\ 
			\left[ \begin{array}{cc} 
				0_{\portsize\times\FOMsize} & \SoBmat^{*}
			\end{array} \right]   
			\left[ \begin{array}{cc} 
				K & 0_{\FOMsize} \\ 
				0_{\FOMsize} & M^{-1} 
			\end{array} \right] x(t)
			,
			\quad x(0)= x_0 =
			\left[ \begin{array}{c} z(0) \\ M \dot{z}(0) \end{array} \right].
		\end{array}
	\end{equation}
	So \eqref{eq:MSD_sys} is rewritten as the \pH system (\ref{eq:LinearPHS}) with the special $J_{2\FOMsize} = {\mathbb J}_{2\FOMsize}$, and block diagonal $R$ and ${{\mathcal H}}$. 
	
	It is easy to see that a \pH system with the structure as in (\ref{eq:SO_pH_FO_form}) is a second order system of (\ref{eq:MSD_sys}). Hence, if our full order \pH system has this structure we would like that this structure remains when constructing the reduced system.  Summarising our problem statement is: For the \pH system (\ref{eq:LinearPHS}), construct a reduced \pH system 
	\begin{equation}{\label{eq:MOR_LinearPHS}}
		\Syssymb_{\text{rLpHS}}\ 
		\colon
		\ \ \ \left\{
		\begin{array}{l}
			\ROMnota{\dot{x}}(t) = (\ROMnota{J} - \ROMnota{R})\ROMnota{{\mathcal H}}\ROMnota{x}(t) + \ROMnota{B}u(t);
			\quad \ROMnota{x}(0)=x_{red,0}; \\
			\ROMnota{y}(t) = \ROMnota{B}^{*}\ROMnota{{\mathcal H}}\ROMnota{x}(t),
		\end{array}
		\right.
	\end{equation}
	that has the following properties:
	\begin{itemize}
		\item The reduced system interpolates the original transfer function in prescribed points, i.e, (\ref{eqn:InterpolationCondition}) is satisfied;
		\item
		The $\ROMnota{J}$ matrix of the reduced system is a canonical Poisson matrix, e.g.,  $\ROMnota{J}=\mathbb{J}_{2\ROMsize}$;
		\item
		The $\ROMnota{R}$ and $\ROMnota{{\mathcal H}}$ keep the block diagonal with first block on the diagonal of in $\ROMnota{R}$ is zero. Furthermore, the upper part of $\ROMnota{B}$ is zero. 
	\end{itemize}
	
	As is clear from \eqref{eq:34}, the proposed symplectic-pH \MOR method ensures that all these requirements are satisfied for the \FOM \eqref{eq:SO_pH_FO_form}. Thus, the reduced order system can be written as a \MSD system again. Another \MOR technique that keeps the \MSD structure is the second-order Arnoldi (\SO-Arnoldi) method. The key idea of the \SO-Arnoldi method is to find a proper projection basis and then apply a projection to the second-order system directly.
	\begin{theorem}[Second-order Arnoldi method \cite{Bai2005}]
		{\label{Th:6}}
		Given a \MSD system \eqref{eq:MSD_sys}, with the interpolation frequencies $s_1, \cdots, s_{\fresize}\in {\mathbb C}$ and tangent directions $\mu_1, \cdots,$ $ \mu_{\fresize}\in {\mathbb C}^{\portsize}$. 
		Let $V_0 \in {\mathbb C}^{\FOMsize\times \ROMsize}$ be such that $V^{*}_0V_0=I$ and $(s_m^2 M +  s_m D + K)^{-1} B\mu_m \in \mathrm{ran}(V_0)$, $m=1, \cdots, {\fresize}$. Then the second order \MSD system
		\begin{equation}
			\label{eq:SO_pH_ROM}
			\begin{array}{l}
				\ROMnota{M} \ROMnota{\ddot{z}}(t) + \ROMnota{D} \ROMnota{\dot{z}}(t) + \ROMnota{K} \ROMnota{z}(t) = \ROMnota{\SoBmat} u(t);
				\quad
				\ROMnota{z}(0) = 
				\ROMnotaZero{z}  ;\\
				y(t) = \ROMnota{\SoBmat}^{*} \ROMnota{\dot{z}}(t)
				,
			\end{array}
		\end{equation}
		with
		\begin{equation}
			\label{eq:const_MDKGamma_SO_Arnoldi}
			\begin{array}{rcl}
				\ROMnota{M}&:=&V^{*}_0MV_0,\
				\ROMnota{D}:=V^{*}_0DV_0,\\
				\ROMnota{K}&:=&V^{*}_0KV_0,\ 
				\ROMnota{\SoBmat}:=V^{*}_0 \SoBmat,\
				\ROMnotaZero{z}:=V^{*}_0z_0
			\end{array}
		\end{equation}
		is a reduced order model of \eqref{eq:MSD_sys} satisfying $G(s_m)\mu_m=\ROMnota{G}(s_m)\mu_m$, $m =1,\ldots, M$, where
		$\ROMnota{G}$ is the transfer function of \eqref{eq:SO_pH_ROM}. 
	\end{theorem}
	
	So we now have two methods for constructing a reduced \MSD system. In the following theorem we show that these methods result in the same reduced system. That is, there exists a basis transformation that transforms the second order reduced system into the symplectic \ROM. To simplify the exposition, we present here the reduced order \MSD model when using the reduction technique from Theorem \ref{th:pH_with_Dissipation}, see also \eqref{eq:34}.
	\begin{equation}
		\label{eq:ROM_symp_MSD}
		\begin{array}{l}
			\ROMnota{\cootrans{M}} \ROMnota{\ddot{\cootrans{z}}}(t) + \ROMnota{\cootrans{K}} \ROMnota{\dot{\cootrans{z}}}(t) + \ROMnota{\cootrans{D}} \ROMnota{\cootrans{z}}(t) = \ROMnota{\cootrans{\Gamma}} u(t);
			\quad
			\ROMnota{\cootrans{z}}(0) = \ROMnotaZero{\cootrans{z}};\\
			\ROMnota{\cootrans{y}}(t) = \ROMnota{\cootrans{\Gamma}}^{*}\ROMnota{\dot{\cootrans{z}}} (t)
			,
		\end{array}
	\end{equation}
	with
	\begin{equation}
		\label{equ:symph_ROM_mats}
		\begin{array}{rcl}
				\ROMnota{\cootrans{M}} &=& (\Psi_1^{*}M^{-1}\Psi_1)^{-1},~
				\ROMnota{\cootrans{K}} = \Phi_1^{*}K\Phi_1,~
				\ROMnota{\cootrans{D}} = \Psi_2^{*}D\Psi_2,~\\
				\ROMnota{\cootrans{\Gamma}} &=& \Psi_2^{*}\Gamma,~
				\ROMnotaZero{\cootrans{z}} = \Psi_2^{*}z_0,
	    \end{array}
	\end{equation}
	and $\Psi_i$, $\Phi_i$, $i=1,2$ given by (\ref{eq:const_phi_psi}).
	\begin{theorem}
		\label{Th:equiv}
		Consider the \MSD system of (\ref{eq:MSD_sys}) with the reduced models 
		(\ref{eq:SO_pH_ROM})--(\ref{eq:const_MDKGamma_SO_Arnoldi}) and (\ref{eq:ROM_symp_MSD})--(\ref{equ:symph_ROM_mats}) based on the same set of $\mu_m$'s and $s_m$'s. There exists a unitary $\Theta \in {\mathbb C}^{k\times k}$ such that $ \ROMnota{z}(t) = \Theta \ROMnota{\cootrans{z}}(t)$.
	\end{theorem} 
	\begin{proof}
		Based on the construction on Theorem \ref{Th:6}, we  construct an orthogonal basis of the following Krylov subspace
		\begin{equation*}
			{\rm span}\left\{(s_1^2M+ D s_1 + K)^{-1} \SoBmat \mu_1,\ldots,(s_m^2M+ D s_m + K)^{-1} \SoBmat \mu_m\right\}.
		\end{equation*}
		and take the columns of $V_0$ equals this basis. In this way $V_0$ satisfies the conditions of the theorem, in particular,  $V_0^*V_0 =I_k$.
		
		Next we turn our attention to the construction of the $\Psi$'s and $\Phi$'s of Subsection \ref{sec:construction of matrices}. Using \eqref{eq:SO_pH_FO_form} we see that 
		\begin{equation}
			\label{eq:44}
			\begin{array}{rcl}
				(s_mI- (J_{2\FOMsize}-R){\mathcal H})^{-1} B \mu_m 
				&=&
				\left[ \begin{array}{c} (s_m^2M+ s_mD + K)^{-1} \SoBmat\mu_m\\ M (s_m^2M+ s_mD + K)^{-1} \SoBmat \mu_m \end{array}\right], \\
				{\mathcal H}(s_m I- (J_{2\FOMsize}-R){\mathcal H})^{-1} B \mu_m 
				&=&
				\left[ \begin{array}{c} K (s_m^2M+ s_mD + K)^{-1} \SoBmat \mu_m\\ (s_m^2M+ s_mD + K)^{-1} \SoBmat \mu_m \end{array}\right].
			\end{array}
		\end{equation}
		
		Take $V_1,V_2\in{\mathbb C}^{2\FOMsize \times \ROMsize}$ as two matrices that satisfy, see (\ref{eq:30}), for all $m \in \{1, \cdots,M\}$
		\begin{equation*}
			(s_m I- (J_{2\FOMsize}-R){\mathcal H})^{-1} B \mu_m\in{\rm ran} (V_1),\quad
			{\mathcal H}(s_m I- (J_{2\FOMsize}-R){\mathcal H})^{-1} B \mu_m\in{\rm ran} (V_2).
		\end{equation*}
		For the partitions of $V_1$ and $V_2$ given as follows
		\begin{equation*}
			V_1 = \left[
			\begin{array}{c}
				V_{11}\\
				V_{12}
			\end{array}
			\right],
			\quad
			V_2= \left[
			\begin{array}{c}
				V_{21}\\
				V_{22}
			\end{array}
			\right].
		\end{equation*}
		we find by applying (\ref{eq:44})  and the construction of $V_0$ that
		$V_{11} = V_0 Q_{11}$, $V_{12}=MV_0Q_{11}$, $V_{21}=KV_0Q_{22}$, and $V_{22}=V_0Q_{22}$, with $Q_{11}$ and $Q_{22}$ non-singular. 
		
		Applying these expressions to the definition of $\Psi_2$ in (\ref{eq:const_phi_psi}) we find that $\Psi_2 = \mathrm{orth}(V_{22}) = \mathrm{orth}(V_0Q_{22})$. Since $V_0^*V_0 = I_k$ this implies that $\Psi_2 = V_0 \Theta$ with $\Theta$ being unitary. 
		
		Using this, we find that the other three matrices in (\ref{eq:const_phi_psi}) are given by:
		\begin{align*}
			\Psi_1 =&\ 
			MV_0Q_{11}(\Psi_2^{*}MV_0Q_{11})^{-1}
			=
			MV_0(V_0^{*}MV_0)^{-1} \Theta,\\
			\Phi_2 =&\ KV_0Q_{22}\left( \Psi_2^{*}KV_0Q_{22}\right)^{-1}
			=
			KV_0\left( V_0^{*}KV_0\right)^{-1} \Theta,
			\\
			\Phi_1 =&\ V_0Q_{11}(\Phi_2^{*}V_0Q_{11})^{-1} = V_0 \left( \Theta^* \left( V_0^{*}KV_0\right)^{-1} V_0^{*}KV_0\right)^{-1} = V_0 \Theta .
		\end{align*}
		Applying these expressions to the model matrices of equation (\ref{equ:symph_ROM_mats}) and using (\ref{eq:const_MDKGamma_SO_Arnoldi}), we obtain
		\begin{align*}
			\ROMnota{\cootrans{K}}=&\ \Theta^{*}V_0^{*}KV_0\Theta = 
			\Theta^{*} \ROMnota{K} \Theta\\
			\ROMnota{\cootrans{D}}=&\ \Theta^{*}V_0^{*}DV_0\Theta = 
			\Theta^{*} \ROMnota{D} \Theta\\
			\ROMnota{\cootrans{\Gamma}} =&\ \Theta^{*}V_0^{*}\Gamma = 
			\Theta^{*} \ROMnota{\Gamma}\\
			\ROMnota{\cootrans{M}}
			=&\
			(\Psi_1^{*}M^{-1}\Psi_1)^{-1}\\
			=&\
			\left(
			\Theta^{*}
			(V_0^{*}MV_0)^{-1}V_0^{*}M^{*}
			M^{-1}
			\
			MV_0(V_0^{*}MV_0)^{-1}
			\Theta
			\right)^{-1}\\
			=&\
			\Theta^{*}
			\left(
			(V_0^{*}MV_0)^{-1}
			\right)^{-1}
			\Theta
			=
			\Theta^{*}
			\ROMnota{M}
			\Theta\\
			\ROMnotaZero{\cootrans{z}} =& 
			\Theta^{*}V_0^*
			z_0
			=
			\Theta^{*}
			\ROMnotaZero{z}
			.
		\end{align*}
		Substituting this in (\ref{eq:ROM_symp_MSD}) we find
		\begin{equation*}
			\begin{split}
				\Theta^* \ROMnota{M} \Theta \ROMnota{\ddot{\cootrans{z}}}(t) &\ + \Theta^*\ROMnota{K} \Theta \ROMnota{\dot{\cootrans{z}}}(t) + \Theta^*\ROMnota{D} \Theta \ROMnota{\cootrans{z}}(t) = \Theta^* \ROMnota{\cootrans{\Gamma}} u(t), \\
				\ROMnota{\cootrans{y}}(t) =&\  \ROMnota{\Gamma}^{*}\Theta \ROMnota{\dot{\cootrans{z}}} (t),
				\quad
				\ROMnotaZero{\cootrans{z}}
				=
				\Theta^*\ROMnotaZero{z}.
			\end{split}
		\end{equation*}
		Thus, $\Sigma_{r,1}$ and $\Sigma_{r,2}$ are equivalent with a coordinate transformation given by $\ROMnota{\cootrans{z}} = \Theta^*\ROMnota{z}$.
	\end{proof}
	
	\section{Numerical Results}	
	\label{sec:RLC}
	
	In this section, we investigate the numerical performance of the symplectic \MOR of \pH-system with dissipation as introduced in Subsection \ref{Subsec:w_dissipation} using the example given by an \LRCR circuit. We first introduce the problem statement of the \LRCR circuit in Subsection \ref{sec:problem_statement}. Then, in Subsection \ref{sec:SO_form}, we introduce the second-order form of this problem and finally compare in Subsection \ref{sec:final_results} the performance of the symplectic \MOR technique with the structure-preserving Arnoldi method (pH-Arnoldi), see \cite{Gugercin2012IRKA} and Theorem \ref{thm:IRKA-pH} and structure-preserving Arnoldi method of second-order systems (\SO-Arnoldi), see \cite{Bai2005} and Theorem \ref{Th:6}.
	
	\subsection{Problem statement}\label{sec:problem_statement}
	
	For the numerical example, we consider an LRCR circuit as presented in Figure \ref{fig:LRCR_circuit}.
	\begin{figure}[H]
		\centering
		\begin{circuitikz}[scale = 0.9][H]
			
			\draw (0,0) to[V=$V_0$] (0,3);
			
			\node at (0, 0) {\Large$\circ$};
			\node at (0, 3) {\Large$\circ$};
			\draw (0,3) to[R=$R_{1,1}$] (1.5,3);
			\draw (1.5,3) to[L=$L_{1}$] (3,3);
			\draw (3,3)to[R = $R_{1,2}$] (3,0);
			\draw (3,3) to (5,3);
			\draw (4.5,3)to[C = $C_{1}$] (4.5,0);
			\draw (5,0)to(0,0);
			\node at (6, 1.5) {\Large$\cdots$};
			
			\draw (6.5,3) to (7,3);
			\draw (7,3) to[R=$R_{\FOMsize
				,1}$] (8.5,3);
			\draw (8.5,3) to[L=$L_{\FOMsize}$] (10,3);
			\draw (10,3)to[R = $R_{\FOMsize,2}$] (10,0);
			\draw (10,3) to (11.5,3);
			\draw (11.5,3)to[C = $C_{\FOMsize}$] (11.5,0);
			\draw (11.5,0)to(6.5,0);
		\end{circuitikz}
		\caption{LRCR circuit}
		\label{fig:LRCR_circuit}
	\end{figure}
	For $i= 1,\ldots,\FOMsize$, we denote the voltage difference on the two sides of the $i-$th capacitor $C_i$ by $V_i$ and denote the electric current that goes through the $i-$th inductor$L_i$ by $I_i$. We arrive at the following equations based on 
	Kirchhoff's current law and Kirchhoff's voltage law
	\begin{equation}
		\label{equ:equality of each components}
		\begin{array}{rcl}
			C_1 \ddt V_1(t) &=& I_1(t) - I_2(t) - R_{1,2}^{-1}V_1(t)\\
			L_1 \ddt I_1(t) &=& V_0(t) - R_{1,1}I_1(t) -V_1(t)\\
			&\vdots&\\
			C_{i-1} \ddt V_{i-1}(t) &=& I_{i-1}(t) - I_{i}(t) - R_{i-1,2}^{-1} V_{i-1}(t)\\
			L_{i-1} \ddt I_{i-1}(t) &=& V_{i-2}(t) - R_{i-1,1}I_{i-1}(t) -V_{i-1}(t)\\
			&\vdots&\\
			C_{\FOMsize} \ddt V_{\FOMsize}(t) &=& I_{\FOMsize}(t) - R_{\FOMsize,2}^{-1}V_{\FOMsize}(t)\\
			L_{\FOMsize} \ddt I_{\FOMsize}(t) &=& V_{\FOMsize-1}(t) - R_{\FOMsize,1}I_{\FOMsize}(t) -V_{\FOMsize}(t).
		\end{array}
	\end{equation}
	For $i=1,\ldots,\FOMsize$, we denote $\phi_i$ as the magnetic flux of $L_i$ and denote $q_i$ as the charge stored in a capacitor $C_i$.
	Then, the system in \eqref{equ:equality of each components} with state given by $x=[q_1,\ldots,q_{\FOMsize},\phi_1,\ldots,\phi_{\FOMsize}]^{\top}$ 
	is
	\begin{equation*}
		\dot{x}(t) = (J_{2\FOMsize}-R){\mathcal H}x(t) + B u(t),
	\end{equation*}
	where 
	\begin{align*}
		J_{2n} &= 
		\left[
		\begin{array}{cc}
			0 & \hat{J}_{\FOMsize}\\
			-\hat{J}_n^{\top} &0
		\end{array}
		\right],
		\quad
		{\rm with}\
		\hat{J}_n\coloneqq
		\left[
		\begin{array}{ccccc}
			1& -1 &&&\\
			&1& -1 &&\\
			&& \ddots &\ddots&\\
			&&&1 &-1\\
			&&&&1  \\
		\end{array}
		\right],\\
		B&=[\underbrace{0,\ldots0}_{\FOMsize},1,0,\ldots,0]^{\top},
		\quad
		{\mathcal H}= 
		{\rm diag}(C_1^{-1},\ldots,C_{\FOMsize}^{-1},L_1^{-1},\ldots,L_{\FOMsize}^{-1}),
		\\
		R &=
		{\rm diag}(	R_{1,2}^{-1},\ldots, R_{\FOMsize,2}^{-1},R_{1,1},\ldots,R_{\FOMsize,1}).&
	\end{align*}
	Since the physical meaning of $I_1 V_0$ is the energy exchange between this system and the environment via the port of the system, it is natural to set the output as $y(t)=
	L_1^{-1}\phi_1=B^{\top}{\mathcal H}x = I_1$. Now, we arrive at a \pH system based on the \LRCR circuit as follows
	\begin{equation}
		\label{equ:sys_pH_sys}
		\begin{array}{rcl}
			\dot{x}(t) &=& (J_{2\FOMsize}-R){\mathcal H} x(t) +B u(t); 
			\\
			y(t) &=& B^{*}{\mathcal H} x(t).
		\end{array}
	\end{equation}
	
	As parameters of the system, we consider the above introduced \LRCR circuit with $\FOMsize =50$, and we set $C_i= 1$\,pF, $L_{i} = 1$\,nH, $R_{i,1} = 1\,$m$\Omega$, and $R_{i,2} = 1\,$k$\Omega$ for $i=1,2,\cdots,N$.  The Bode plot of this system is presented in Figure \ref{fig:LRCR_bode} which shows that the system has relatively higher response at the frequency below $1$\,GHz, and the behaviour of the system oscillates a lot at the frequency around $10$\,GHz.
	
	\begin{figure}[H]
		\begin{subfigure}{0.48\linewidth}
			\begin{tikzpicture}
				\begin{axis}[
					width = 0.95\linewidth,
					height = 0.6\linewidth,
					grid=major,
					xmode=log, 
					xlabel={Frequency (GHz)},
					ylabel={Mag},
					legend pos=north east,
					xmin= 0.01,
					xmax=100,
					]
					\addplot [mark=none, blue] 
					table[x index=0, y index=1, col sep=comma] 
					{bode_mag_c_1e-3_r1_1.dat};
				\end{axis}
			\end{tikzpicture}
		\end{subfigure}
		\begin{subfigure}{0.48\linewidth}
			\begin{tikzpicture}
				\begin{axis}[
					width = 0.95\linewidth,
					height = 0.6\linewidth,
					grid = major,
					xmode=log, 
					xlabel={Frequency (GHz)},
					ylabel={Phase},
					legend pos=north east,
					xmin= 0.01,
					xmax=100,
					]
					\addplot [mark=none, blue] 
					table[x index=0, y index=1, col sep=comma] 
					{bode_phase_c_1e-3_r1_1.dat};
				\end{axis}
			\end{tikzpicture}
		\end{subfigure}
		
		\caption{Magnitude (left) and phase (right) of the transfer function (\FOM) corresponds to the \LRCR circuit where $C_i= 1$\,pF, $L_{i} = 1$\,nH, $R_{i,1} = 1$\,m$\Omega$, and $R_{i,2} = 1\,$k$\Omega$.}
		\label{fig:LRCR_bode}
	\end{figure}
	
	\subsection{Second-order form of the LRCR circuit} \label{sec:SO_form}
	
	In this work, we compare the proposed algorithm with the \SO-Arnoldi method, thus, we reformulate \eqref{equ:sys_pH_sys} as a second-order system.
	Based on the Kirchhoff's current law, we have
	\begin{equation}
		\label{equ:replace_I_via_V}
		\begin{array}{rcl}
			I_1(t)  &=& \sum_{j=1}^{\FOMsize}\frac{V_j(t) }{R_{j,2}} + \sum_{j=1}^{\FOMsize}C_j\ddt V_j(t) \\
			&\vdots&\\
			I_i(t)  &=& \sum_{j=i}^{\FOMsize}\frac{V_j(t) }{R_{j,2}} + \sum_{j=i}^{\FOMsize}C_j\ddt V_j(t) \\
			&\vdots&\\
			I_{\FOMsize}(t)  &=& \frac{V_{\FOMsize}(t) }{R_{\FOMsize,2}} + C_{\FOMsize}\ddt V_{\FOMsize}(t).
		\end{array}
	\end{equation}
	By inserting \eqref{equ:replace_I_via_V} into \eqref{equ:equality of each components}, and denoting $z(t)  = [V_1(t) ,V_2(t) ,\ldots,V_{\FOMsize}(t) ]^{\top}$, we get the second-order equation
	\begin{equation}
		\label{equ:LRCR_SO_form}
		\begin{array}{rcl}
			M\ddot{z}(t) + D\dot{z}(t) + Kz(t)
			&=&
			[1,0,\ldots,0]^{\top}
			V_0(t),\\
			y(t) &=& R_2^{*} z(t) + C^{*}\dot{z}(t),
		\end{array}
	\end{equation}
	where
	\begin{equation*}
		M = {\rm triu}( L C^{\top}),\quad
		D =  {\rm triu}\left( L R_2^{\top} + R_1C^{\top}\right),\quad
		K =	\left(	{\rm triu}(	R_1R_2^{\top})+\hat{J}_{\FOMsize}^{\top}\right),
	\end{equation*}
	with $L = [L_1,L_2,\ldots,L_{\FOMsize}]^{\top}$, $R_1 = [R_{1,1},R_{2,1},\ldots,R_{\FOMsize,1}]^{\top}$, $C = [C_1,C_2,\ldots,C_{\FOMsize}]^{\top}$, $R_2 = [R_{1,2},R_{2,2},\ldots,R_{\FOMsize,2}]^{\top}$, and the operator ``${\rm triu}$" takes the upper-triangular part of a square matrix.
	It needs to be highlighted that although matrices ${\rm triu}( L C^{\top}), {\rm triu}\left( L R_2^{\top} + R_1C^{\top}\right)$ and $	\left({\rm triu}(R_1R_2^{\top})+\hat{J}_{\FOMsize}^{\top}\right)$ are not symmetric, the system given by \eqref{equ:LRCR_SO_form} is of \pH form. In this work, only the $H_2$ and $H_{\infty}$ errors are discussed. Thus, the initial condition is not needed in the computation.
	
	\subsection[Comparison of H 2 and H inf error]{Comparison of ${H_2}$ and ${H_{\infty}}$ error}
	\label{sec:final_results}
	
	Based on the Bode plot, we select interpolation points within the range $[10^{-2},10^{2}]$ as this is the interesting interval. Moreover, in this interval, we choose interpolation points uniformly in the log scale. For instance, if we need 6 interpolation points, we choose
	$[10^{-2}i, -10^{-2}i, 10^{0}i, -10^{0}i, 10^{2}i, -10^{2}i]$, where $i$ is the complex unit.\footnote{We select interpolation points in conjugate pairs, which ensures that the resulting reduced order models are real systems.} We construct \ROMs by using the \SO-Arnoldi method, the \pH-Arnoldi method and the symplectic-\pH method, with the same interpolation points. If the \ROM is stable, then its $H_2$ and $H_{\infty}$ errors are computed by
	\begin{equation*}
		\label{equ:H_2_H_inf_error}
		\begin{array}{l}
			H_2 {\rm error} (\ROMnota{G})
			=
			\sqrt{\frac{1}{2\pi}
				\int_{-\infty}^{\infty}
				{\rm Trace}
				\left(
				(G(i\omega)-\ROMnota{G}(i\omega))^{*}
				(G(i\omega)-\ROMnota{G}(i\omega))
				\right)
				d\omega
			},
			\quad
			\\
			H_{\infty} {\rm error} (\ROMnota{G})
			=
			\max_{\omega\in {\mathbb R}}
			\|G(i\omega)-\ROMnota{G}(i\omega)\|_2.
		\end{array}
	\end{equation*}
	The results are presented in Figure \ref{fig:LRCR_SO_c_1e-3_r1_1e-2}. 
	
		\begin{figure}[H]
		\centering
		\begin{subfigure}{0.48\linewidth}
			\begin{tikzpicture}
				\begin{axis}[
					width=0.95\linewidth,
					height=0.65\linewidth,
					ymode=log,
					grid=major,
					ylabel={$H_2$ error},
					xlabel={Reduced-order (r)},
					xtick={8,16,24,32},    
					ytick={0.1548, 0.25119, 0.39811, 0.63096},
					yticklabels={$10^{-0.8}$, $10^{-0.6}$, $10^{-0.4}$, $10^{-0.2}$},
					xmin=7.5, xmax=36.5,
					ymin = 0.1548, ymax= 0.63096,
					unbounded coords=jump, 
					]
					\addplot [mark=star ,blue, mark size= \MarkSize, line width=\PlotLineWidth] 
					table[x index=0, y index=1, col sep=comma] 
					{error_H2_foarnoldi.dat};
					\addplot [mark=o ,red, mark size= \MarkSize, line width=\PlotLineWidth] 
					table[x index=0, y index=1, col sep=comma] 
					{error_H2_soarnoldi.dat};
					\addplot [mark=square ,cyan, mark size= \MarkSize, line width=\PlotLineWidth] 
					table[x index=0, y index=1, col sep=comma] 
					{error_H2_symph.dat};
				\end{axis}
			\end{tikzpicture}
			\caption{$H_2$ error}
		\end{subfigure}
		\begin{subfigure}{0.48\linewidth}
			\begin{tikzpicture}
				\begin{axis}[
					width=0.95\linewidth,
					height=0.65\linewidth,
					ymode=log,
					grid=major,
					ylabel={$H_{\infty}$ error},
					xlabel={Reduced-order (r)},
					legend columns=3,
					legend entries={pH-Arnoldi, SO-Arnoldi, Symplectic-pH},
					legend style={font=\footnotesize},
					legend to name= Legend_Fig_FO_Comp,
					unbounded coords=jump, 
					xmin=7.5, xmax=36.5,
					ymin=1e-1, ymax=1.8e0,
					xtick={8,16,24,32},    
					ytick={0.1548, 0.39811, 1},
					yticklabels={$10^{-0.8}$, $10^{-0.4}$, $10^{0}$},
					]
					\addplot [mark=star ,blue, mark size= \MarkSize, line width=\PlotLineWidth] 
					table[x index=0, y index=1, col sep=comma] 
					{error_Hinf_foarnoldi.dat};
					\addplot [mark=o ,red, mark size= \MarkSize, line width=\PlotLineWidth] 
					table[x index=0, y index=1, col sep=comma] 
					{error_Hinf_soarnoldi.dat};
				    \addplot [mark=square ,cyan, mark size= \MarkSize, line width=\PlotLineWidth] 
					table[x index=0, y index=1, col sep=comma] 
					{error_Hinf_symph.dat};
				\end{axis}
			\end{tikzpicture}
			\caption{$H_{\infty}$ error}
		\end{subfigure}
		\ref{Legend_Fig_FO_Comp}
		\caption{$H_2$ (left) and $H_{\infty}$ (right) errors of \ROMs obtained by different structure-preserving \MOR methods.}
		\label{fig:LRCR_SO_c_1e-3_r1_1e-2}
	\end{figure}
	
	Compared to the \pH-Arnoldi, the symplectic-\pH provides \ROMs with lower $H_2$ and $H_{\infty}$ error in most cases. 
	Although \SO-Arnoldi and symplectic-\pH have very close performance when the reduced-order is below 24, the \SO-Arnoldi has bad performance in the $H_2$ and $H_{\infty}$ error when the reduced order is 32. Nevertheless, the reduced order system resulting from \SO-Arnoldi is stable at reduced dimension 32. However, at reduced dimensions 28 and 36, the \SO-Arnoldi method returns an unstable system, such that the corresponding values are missing in Figure \ref{fig:LRCR_SO_c_1e-3_r1_1e-2}. The poles of these two systems are shown in Figure \ref{fig:LRCR_SO_poles}.
    \begin{figure}[H]
    	\centering
    	\begin{subfigure}{0.48\linewidth}
    		\begin{tikzpicture}
    			\begin{axis}[
    				width=0.95\linewidth,
    				height=0.65\linewidth,
    				grid=none,
    				xlabel={Real},
    				ylabel={Imag},
    				legend pos=north east,
    				]
    				\addplot [only marks,  mark=* ,blue, mark size=0.5*\MarkSize] 
    				table[x index=0, y index=1, col sep=comma] 
    				{LRCR_soarnoldi_poles_c_1e-3_r1_7.dat};
    				\draw[dashed, thick, red] (axis cs:0,-100) -- (axis cs:0,100);
    			\end{axis}
    		\end{tikzpicture}
    		\caption{r =28 }
    	\end{subfigure}
    	\begin{subfigure}{0.48\linewidth}
    		\begin{tikzpicture}
    			\begin{axis}[
    				width=0.95\linewidth,
    				height=0.65\linewidth,
    				grid=none, 
    				xlabel={Real},
    				ylabel={Imag},
    				legend pos=north east,
    				]
    				\addplot [only marks,  mark=* ,blue, mark size=0.5*\MarkSize] 
    				table[x index=0, y index=1, col sep=comma] 
    				{LRCR_soarnoldi_poles_c_1e-3_r1_9.dat};
    				\draw[dashed, thick, red] (axis cs:0,-100) -- (axis cs:0,100);
    			\end{axis}
    		\end{tikzpicture}
    		\caption{r =36}
    	\end{subfigure}
    	\caption{Poles of \ROMs obtained by the \SO-Arnoldi method with reduced dimension $r=28$ (left) and $r=36$ (right).  
    				}
    	\label{fig:LRCR_SO_poles}
    \end{figure}
	
	\section{Conclusion}
	\label{sec:conclusion}
	
	In this paper, we presented symplectic-\pH, a \MOR method which, next to maintaining the \pH structure in the \ROM, is able to maintain additional block structure in the system matrices. We proved that this method gives the same \ROM as the \SO-Arnoldi method for \MSD systems. However, for the class of electrical networks the method provides a novel \MOR technique. Blindly applying the \SO-Arnoldi method might lead to an unstable system as is numerically shown in the example. 
	
	The results presented in this paper focus on interpolating the transfer function at given frequency data. To ensure a better fit between the transfer function of the \FOM and the \ROM, one aspect of future work is to match the derivatives of the transfer function.  
	
\bibliographystyle{abbrvurl}
\bibliography{references}

\end{document}